\documentclass[letterpaper, 10 pt, conference]{ieeeconf}  % Comment this line out if you need a4paper

\IEEEoverridecommandlockouts                              % This command is only needed if 
\usepackage{amsmath,amssymb,amsfonts,fouridx}
\usepackage{graphicx} % Required for the inclusion of images
\usepackage{mathtools}
\usepackage{cite}
\usepackage{enumerate}
\usepackage[usenames, dvipsnames]{color}
\usepackage[utf8]{inputenc}
\newtheorem{theorem}{Theorem}  
\usepackage{subfigure}

\newcommand*{\QEDA}{\hfill\ensuremath{\blacksquare}}%
\newtheorem{ass}{Assumption}                                                %
\newtheorem{lemma}{Lemma}                                              %
\newtheorem{rem}{Remark}                                                  %
\newtheorem{coro}{Corollary} 
\newtheorem{define}{Definition}                                             %

\newcommand{\mf}{\mathfrak}

\newcommand{\wt}{\widetilde}
\title{\LARGE \bf Partial containment control over signed graphs*\thanks{*Accepted for presentation at the 2019 European Control Conference.}} % Title
\author{Pietro DeLellis$^\ast$\thanks{$^\ast$Department of Electrical Engineering and Information Technology, University of Naples Federico II, 80125, Naples, Italy.}, Anna DiMeglio$^\ast$%
$^\dagger$\thanks{$^\dagger$Corresponding author; e-mail: anna.dimeglio@unina.it} ,  Franco Garofalo$^\ast $
 and Francesco Lo Iudice$^\ast $
}
\begin{document}
\maketitle
%\begin{frontmatter}

\begin{abstract}
In this paper, we deal with the containment control problem in presence of antagonistic interactions. In particular, we focus on the cases in which it is not possible to contain the entire network due to a constrained number of control signals. In this scenario, we study the problem of selecting the nodes where control signals have to be injected to maximize the number of contained nodes. Leveraging graph condensations, we find a suboptimal and computationally efficient solution to this problem, which can be implemented by solving an integer linear problem. The effectiveness of the selection strategy is illustrated through representative simulations.
\end{abstract}

%\end{frontmatter}
%----------------------------------------------------------------------------------------
%	SECTION 1
%----------------------------------------------------------------------------------------

\section{Introduction}
The problem of coordinating the dynamics of ensembles of agents connected through static or dynamic network topologies has been deeply investigated in the last decades. In particular, departing from the pioneering work of DeGroot in the Seventies \cite{de74}, substantial research effort has been devoted to unravel the mechanisms leading to the emergence of consensus in networks of simple integrators. The problem has
been thoroguhly studied both in continuous and discrete time \cite{olmu04}, on undirected or directed graphs, and in presence of delays \cite{olfa07}. 
Consensus has been also investigated in a
leader-following setting, in which one node, the leader, drives a network of linear systems towards a desired value \cite{nich10}.
Achieving consensus is not the only possible control goal in multi-agent systems. Indeed, in applications of networks of autonomous agents, the objective is often to \emph{contain} a group of agents within a certain area, e.g. not to enter populated areas. Motivated by that, Ji and coworkers introduced the so-called containment control problem, where multiple leaders have to drive a group of mobile agents within a desired convex polytope \cite{jife08}. Later works have further analyzed this problem to account for the presence of directed interactions \cite{care09}, possible switches in the network topology \cite{care12,dedi17}, uncertainty \cite{pewa15}, and higher-order dynamics \cite{cast11,lixi12}.

As noted by Altafini in \cite{al13}, most of the works on consensus and containment control relies on the assumption of cooperation among the agents in the system, as all the network edges are assumed to have positive weights. However, in social network theory, besides cooperative interactions, also antagonism is commonly observed \cite{faal11,wafa94}. A natural setting to describe such interactions is to characterize the network topology through the so-called \emph{signed graphs}, introduced in the Fifties by Harary \cite{ha59} to model the
disliking, indifference, and liking sentiments described by psychologists in social interactions. These considerations motivated a bulk of studies on consensus and containment control over signed graphs \cite{al13,huzh13,lich15,shpr15,xica16,medu16,me17,zuso18}. In particular, in \cite{al13} the author showed that when the graph is \emph{balanced}, bipartite consensus can be achieved, that is, the states of all the agents will have the same modulus, but possibly different signs. These results were later extended to the case of non-strongly connected graphs and discrete-time dynamics \cite{huzh13, lich15, shpr15,xica16,medu16}.  
%Some other investigations have been performed relaxing the assumptions on the connectivity of the signed graph as in \cite{xica16} while in \cite{medu16} the bipartite consensus is studied on signed digraphs that only have a spanning tree. The authors find that the existence of the spanning tree condition is sufficient to guarantee the interval bipartite consensus, that is, the convergence of the absolute value of the nodes' state to the modulus of the converged values of the root agents.  
%
Recently, a first definition of containment control over signed graphs was given in \cite{me17}. Specifically, the author says that a network is contained when the states of its nodes converge towards the convex hull spanned by the leaders and by their symmetric trajectories. Assuming continuous-time dynamics, conditions guaranteeing the achievement of full network containment were achieved. Similar results were obtaibed in \cite{zuso18} for the case of generic linear heterogeneous node dynamics. 

However, in large directed networks where the number of control signals is constrained, it is seldom possible to contain the whole network, as it would require to directly inject a control input in every root strongly connected component (RSCC) of the network. Therefore, 
in this paper, we formulate the partial containment control problem over signed graphs. When the number of control inputs is limited, our goal is to maximize the number of nodes we asymptotically contain. Exploiting two graph condensations, we first derive sufficient conditions to contain the atom of our network, that is, a SCC, and then devise a suboptimal algorithm to efficiently deploy the control inputs. Interestingly, this strategy only relies on information on the network topology and not on the initial state of the nodes, which might be non-accessible. Moreover, we illustrate how the algorithm can be translated into an integer linear program, and we illustrate its effectiveness on a representative numerical testbed.

\section{Mathematical Preliminaries}\label{sec:prel}
\subsection{Signed graphs}
A directed signed graph $\mathcal{G}$ consists of an unsigned directed graph $\mathcal{U}=\{\mathcal{V},\mathcal{E}\}$ and a partial mapping $\sigma:\mathcal{E}\rightarrow\{+,-\}$\cite{za82}. An edge $(i,j)\in\mathcal{E}$ is called positive if $\sigma(i,j)=\{+\}$, while it is called negative otherwise.
We associate to $\mathcal G$ a weighted adjacency matrix $A$, whose $ij$-th element $a_{ij}$ is positive if $(i,j) \in \mathcal{E} \land \sigma(i,j)=\{+\} $, negative if $(i,j) \in \mathcal{E} \land \sigma(i,j)=\{-\} $, and zero otherwise.

Throughout the manuscript, we shall consider signed graphs fulfilling the following assumption.
\begin{ass}\label{ass:adjacency}
$\left| a_{ii}\right|>0 $ and $\sum_{j=1}^n \left|a_{ij}\right|=1$, for all $i=1,\ldots,n$. 
\end{ass}
\begin{define}\label{structural balance}
A directed signed graph $\mathcal{G}$ is structurally balanced if there exists a bipartition $\lbrace \mathcal{V}^1, \mathcal{V}^2\rbrace$ of $\mathcal{V}$, 
such that $a_{ij}\geq 0$, for all $(i,j)\in \mathcal{V}^\theta$ and $a_{ij}\le 0$ for all $(i\in\mathcal{V}^\theta , j\in\mathcal{V}\setminus\mathcal{V}^\theta)$, for all $\theta\in\{1,2\}$. $\mathcal{G}$ is unbalanced otherwise.
\end{define}
Notice that every unsigned graph is structurally balanced with $\mathcal{V}^1=\mathcal{V}$ and $\mathcal{V}^2=\emptyset$.
Following \cite{xica16} and \cite{he14}, we define the \emph{enlarged graph} associated to $\mathcal{G}$ as follows:
%\begin{figure}
%\centering
%\includegraphics[width=10cm]{enlarged_graph}
%\caption{Enlarged graph $\widetilde{\mathcal{G}}$ (right panel) associated to a sample graph $\mathcal{G}$ (left panel).}\label{fig:enlarged_graph}
%\end{figure}
\begin{define}\label{def:enlarged_graph}
The enlarged graph $\mathcal{\widetilde{G}}=\{\widetilde{\mathcal{V}},\widetilde{\mathcal{E}} \}$ associated to $\mathcal{G}$ is a (unsigned) directed graph of $2n$ nodes ($\widetilde{\mathcal{V}}=\{1,\ldots,n,1^-,\ldots,n^- \}$) and all positive edges related to that of $\mathcal{G}$ through the adjacency matrix $\widetilde{A}$, whose elements are 
\begin{equation*}\label{en_graph}
\begin{aligned}
\tilde{a}_{ij}= \tilde{a}_{i+N,j+N}=& \max (0,a_{ij})\geq 0,\\
\tilde{a}_{i+N,j}=\tilde{a}_{i,j+N}=& \max(0, -a_{ij}) \geq 0,
\end{aligned}
\end{equation*}
for $i,j=1,\ldots,N$.
\end{define}
%In what follows, we shall denote by $\mathbbm{1}$ a unit vector of appropriate length. Furthermore, given $z(k):\mathbb{Z}^+\rightarrow \mathbb{R}$, if $\displaystyle \lim_{k\rightarrow +\infty} z(k)$ exists, then we denote it by $\bar{z}$. 
\subsection{Some useful lemmata}
\begin{lemma}\label{lem:red_norm_form}
Let us consider a reducible matrix in normal form:
\begin{equation*}\label{eq:red_norm_form}
M=\left[ {\begin{array}{*{20}{c}}
{{M_{1}}}&0&0&0\\
 \vdots & \ddots & \ddots & \vdots \\
0& \cdots &{{M_{q}}}&0\\
{{R_1}}& \cdots &{{R_{q}}}&{{S}}
\end{array}} \right],
\end{equation*}
where $M_{j}$, $j=1,\ldots,q$, are semi-convergent irreducible matrices and $S$ is a convergent matrix. We then have
\begin{equation*}\label{eq:red_asympt}
\lim_{k\rightarrow+\infty}M^k=\left[ {\begin{array}{*{20}{c}}
{{M_{1}^{\infty}}}&0&0&0\\
 \vdots & \ddots & \ddots & \vdots \\
0& \cdots &{{M_{q}^{\infty}}}&0\\
{{\cdots}}& R^*_j &{\cdots}&{0}
\end{array}} \right],
\end{equation*}
where $M_{j}^{\infty}=\lim_{k\rightarrow+\infty}M_j^k$ and $R^*_{j}=(I-S)^{-1} R_j M_j^{\infty}$. Furthermore, if $\lambda=1$ is an eigenvalue of $M_j$, then $M_j^{\infty}=\psi_j\xi_j^T$, where $\xi_j$ and $\psi_j$ are the left and right eigenvectors associated to $\lambda=1$, respectively, scaled so that $\xi_j^T\psi_j=1$. 
\end{lemma}
\begin{lemma}\label{lem:str_bal}
\cite{xica16} Given a strongly connected signed graph $\mathcal{G}$ and its associated enlarged graph $\mathcal{\widetilde{G}}$, 
$\mathcal{G}$ is structurally balanced if and only if $\widetilde{\mathcal{G}}$ is disconnected and composed of two strongly connected components.
\end{lemma}
\begin{lemma}\label{lem:str_unbal}
\cite{xica16} Given a strongly connected signed graph $\mathcal{G}$ and its associated enlarged graph $\mathcal{\widetilde{G}}$, $\mathcal{G}$ is structurally unbalanced if and only if $\mathcal{\widetilde{G}}$ is strongly connected.
\end{lemma}
\subsection{Graph condensations}\label{subsec:cond}
\begin{define}\label{def:cond_fun}
Given any pair of vertex sets $\{\mathcal{V},\mathcal{V}'\}$, with $\left|\mathcal{V}\right|\ge \left|\mathcal{V}'\right|$, any (single-valued) function $f:\mathcal{V}\rightarrow\mathcal{V}'$ is called a {\normalfont condensing function}. Moreover,
$$\mathcal{V}_i:=\{ t\in\mathcal{V}:f(t)=i \},\quad i\in\mathcal{V}'.
$$
\end{define}
\begin{define}\label{def:classical_cond}
Let us consider a graph $\mathcal{G} = (\mathcal{V},\mathcal{E})$, a vertex set $\mathcal{V}'$, a condensing function $f:\mathcal{V}\rightarrow\mathcal{V}'$, and the edge set $\mathcal{E}'=\{(i\in\mathcal{V}',j\in\mathcal{V}'),i\ne j|\exists(t,u)\in\mathcal{E} |f(t)=i,f(u)=j \}$. The graph $\mathcal{G}'=\{\mathcal{V}',\mathcal{E}'\}$ is the condensation of $\mathcal{G}$ induced by $f$.
\end{define}
\begin{define}\label{def:novel_cond}
The classic condensation $\mathcal{G}^c$ of a graph $\mathcal{G}$ is the condensation of $\mathcal{G}$ induced by the condensing function $f^c$ that associates each node of $\mathcal{G}$ to the strongly connected component it belongs to.
\end{define}

Now, we introduce a novel condensation of a graph $\mathcal{G}$, denoted as the {\it signed condensation} $\mathcal{G}^s$ of $\mathcal{G}$:
\begin{define}\label{def:signed_cond}
The signed condensation $\mathcal{G}^s$ of a graph $\mathcal G$ is the condensation of $\mathcal{G}$ induced by the condensing function $f^s$ that associates each node of $\mathcal{G}$ to the strongly connected component of $\widetilde{\mathcal{G}}$ it belongs to.
\end{define}
Notice that if $\mathcal{G}$ is unsigned, then $\mathcal{G}^c= \mathcal{G}^s$. The correspondences between the diverse condensations are illustrated in Figure \ref{fig:condensations}. Moreover, we observe that the signed condensation $\mathcal{G}^s$ is a directed acyclic graph. From now on, we call directed acyclic condensation every condensation that is a directed acyclic graph (DAG). We can now give the following definition.
\begin{figure}
\includegraphics[width=\columnwidth]{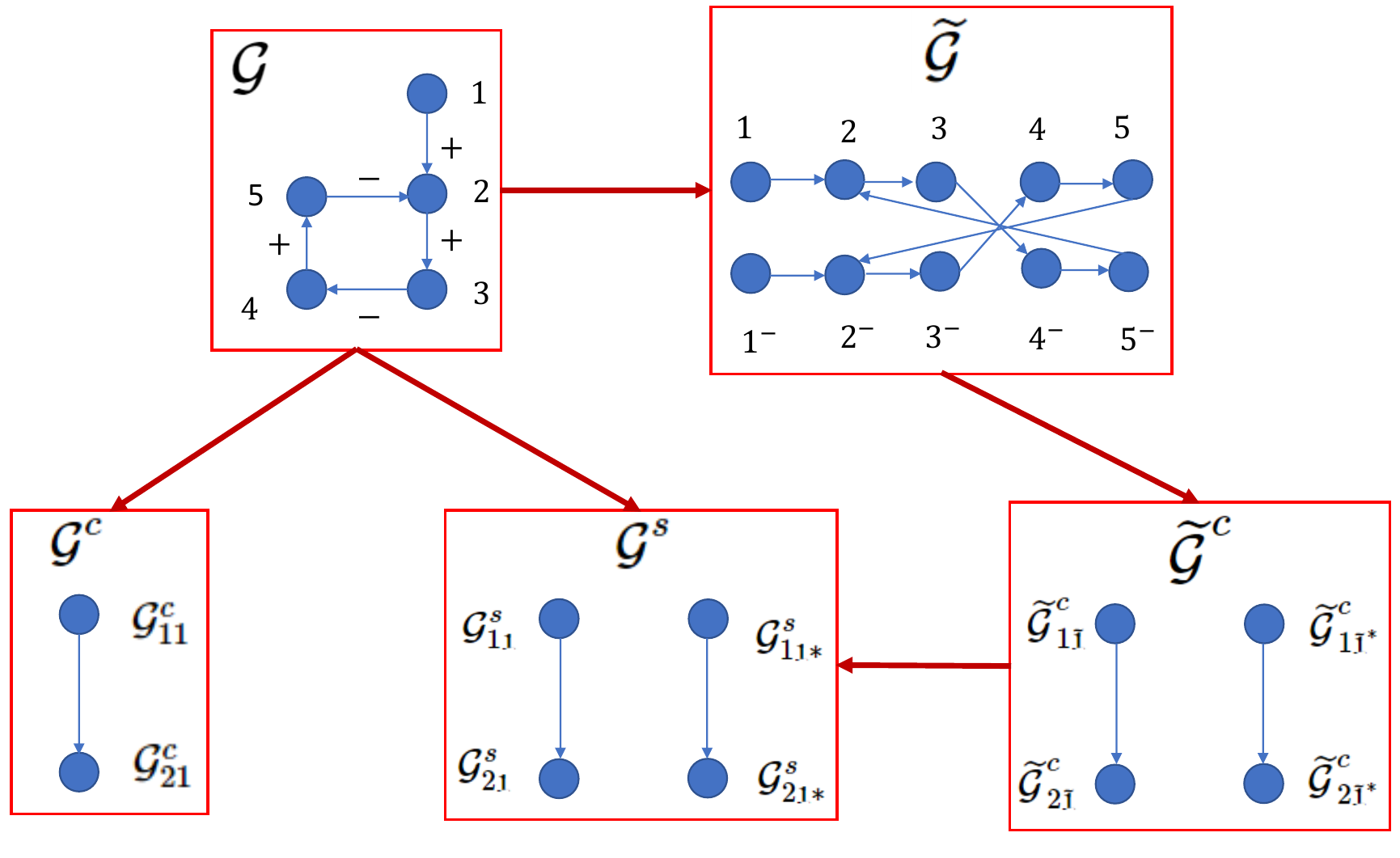}
\caption{The correspondences between the various condensations and their decomposition in levels are illustrated with reference to a sample signed graph $\mathcal{G}$. }\label{fig:condensations}
\end{figure}
\begin{define}\label{def:dec_levels}
Let us consider a directed acyclic condensation $\mathcal{G}^d = (\mathcal{V}^d, \mathcal{E}^d)$ of a graph $\mathcal{G}=(\mathcal V,\mathcal E)$ induced by a condensing function $f^d$.
Node $i\in \mathcal{V}^d$ belongs to level $1$ if $\nexists j: (j,i) \in \mathcal{E}^d$. Furthemore, 
a node $i\in \mathcal{V}^d$ belongs to the $l(>1)$-th level of $\mathcal{V}^d$ if
$$\forall (j,i)\in \mathcal{E}^d ;\ j\in  \mathrm{level}  \ p<l.$$
Moreover, the total number of levels is denoted by $\ell_d$ and the number of nodes of  $\mathcal{G}^d$ in a given level $l$ is $n_l^d$.  
\end{define}
Given a directed acyclic condensation $\mathcal{G}^d$ of $\mathcal{G}$, we  associate to each node in $\mathcal{V}^d$ a pair of indexes $(a,b)$: the first will indicate the level the node belongs to and the second its (random) ranking in that level, see Figure \ref{fig:condensations}. Accordingly, we can define a function $g^d$ that associates to each $(a,b)$ the corresponding node of $\mathcal{V}^d$. Now, we can partition (and sort) the set of nodes of $\mathcal{V}$ as
\begin{equation}
\label{eq:generic_sorting}
\mathcal{V}=\{\mathcal{V}_{11}^d,\ldots,\mathcal{V}_{1n_1^d}^d,\ldots,\mathcal{V}_{\ell_d n_\ell^d}^d\},
\end{equation}
where
\begin{equation}
\label{eq:generic_subgraph_level}
\mathcal{V}_{ij}^d:=\{t\in\mathcal{V}:f^d(t)=g^d(i,j)\}.
\end{equation}
Moreover, we denote $\mathcal{G}_{ij}^d\subseteq\mathcal{G}$ the subgraph induced by $\mathcal{V}_{ij}^d$. Consequently, we indicate by $\mathcal{G}_{ij}^c$ a strongly connected component (SCC) of $\mathcal{G}$, for all $i=1,\ldots,\ell_c$, $j=1,\ldots,n_i^c$, and that $\mathcal{G}^c_{11},\ldots,\mathcal{G}^c_{1h}, \ldots, \mathcal{G}_{1 n_1^c}$ are its $n_1^c\ge 1$ aperiodic root strongly connected components (RSCCs). Notice that the number of levels of $\mathcal{G}^c$, $\widetilde{\mathcal{G}}^c$ and $\mathcal{G}^s$ is the same, that is, $\ell_c=\tilde{\ell}_c=\ell_s:=\ell$. Moreover, for all $l = 1,\dots,\ell$, $n_l^c \leq n^s_l \leq \tilde{n}_l^c$. Moreover, any SCC of $\mathcal{G}$ can be classified as of 
\begin{enumerate}
\item {\it type 1} if it has no negative weights;
\item {\it type 2} if it has at least one negative weight and is structurally balanced;
\item {\it type 3} if it has at least one negative weight and is structurally unbalanced.
\end{enumerate}
\begin{rem}\label{rem:associations}
For all $h=1,\ldots,n_l$, $l=1,\ldots,\ell$, we associate to the $h$-th node of the $l$-th level of $\mathcal{G}^c$
\begin{itemize}
\item the $\mf h$-th ($\tilde{\mf h}$-th) and the $\mf h^*$-th ($\tilde{\mf h}^*$-th) nodes of the $l$-th level of $\mathcal{G}^s$ ($\widetilde{\mathcal{G}}^c$) such that  
$\mathcal{V}^c_{lh}=\mathcal{V}^s_{l \mf h}\cup\mathcal{V}^s_{l \mf h^*}\subset\widetilde{\mathcal{V}}^c_{l\tilde{\mf h}}\cup\widetilde{\mathcal{V}}^c_{l\tilde{\mf h}^*},$
if $\mathcal{G}^c_{lh}$ is of type 1 or type 2;
\item the $\mf h$-th ($\tilde{\mf h}$-th) node of the $l$-th level of $\mathcal{G}^s$ such that  
$\mathcal{V}^c_{lh}=\mathcal{V}^s_{l \mf h} \subseteq \widetilde{\mathcal{V}}^c_{l\tilde{\mf h}},$
if $\mathcal{G}^c_{lh}$ is of type 3.
\end{itemize}
Moreover, we associate to the $\tilde{\mf h}$-th node of the $l$-th level of $\widetilde{\mathcal{G}}^c$ the $h$-th node of the $l$-th level of $\mathcal{G}^c$ such that
$$\widetilde{\mathcal{V}}_{l \tilde{\mf h}}^c\cap\mathcal{V}_{lh}^c\ne\emptyset. $$
\end{rem}
These associations between the nodes of the condensations are clearly illustrated in Figure \ref{fig:condensations}.
\section{Problem formulation}
Let us consider a signed graph $\mathcal G$ with $n$ nodes and associated weighted adjacency matrix $A$. 
A subset $\mathcal{C}\subset \mathcal{V}$ comprises the $m$ \emph{leaders} (sometimes also denoted pinners \cite{liwa04, podi08, tude14, liba16}, depending on the context), that is, nodes that have no incoming links.
Denoting $x_i\in\mathbb{R}$ the state of the $i$-th node, the dynamics over the signed graph $\mathcal{G}$ are described by
\begin{equation}\label{eq:signed_uncontrolled}
x_i(k+1)=x_i(k)+\sum_{j =1}^N  a_{ij} \left(x_j(k)-\mathrm{sign}(a_{ij})x_i(k)\right), 
\end{equation}
for all $i=1,\ldots,n$, or, equivalently, by
\begin{equation*}
x(k+1)=Ax(k),
\end{equation*}
where $x=[x_1,\ldots,x_n]^T$ is the vector of the nodes' states. 

In this paper, we focus on the case in which
\begin{equation*}
\left|a_{ij}\right|=\left\{\begin{array}{ll}
\frac{1}{\left|\mathcal{N}_i\right|} & \text{if } j\in\mathcal{N}_i,\\
0 & \text{otherwise,} 
\end{array}\right.
\end{equation*}
with $\mathcal{N}_i=\{j\in\mathcal{V}:(j,i)\in\mathcal{E}\}$ being the set of neighbors' of $i$,
but the results given in the following can be easily extended to alternative rules for computing $a_{ij}$ that are consistent with Assumption \ref{ass:adjacency}. 

Notice that $x_i(k+1)=x_i(k)=x_i(0)$ for all $i\in\mathcal{C}$.
From \cite{me17}, we give the following definition of containment in signed graphs.
\begin{define}
A node $i\in\mathcal{V}-\mathcal{C}$ is asymptotically signed contained when
\begin{equation}\label{eq:cont_s_i}
\limsup_{k\rightarrow+\infty}\left| x_i(k)\right|\le \max_{j\in\mathcal{C}} \left|x_j(0)\right|,
\end{equation}
\end{define}
\begin{define}\label{def:partial}
Network \eqref{eq:signed_uncontrolled} is $q$-partially signed contained if there exist a subset $\mathcal{Q}\subseteq\mathcal{V}$ of cardinality $q$ such that all the nodes in $\mathcal{Q}$ are asymptotically contained. If $q=n-m$, then network \eqref{eq:signed_uncontrolled} is signed contained.
\end{define}
Let us denote by $\mathcal L$ the set of nodes directly controlled by the leaders, that is,
$$ \mathcal L=\left\{ i\in\mathcal V\ |\ \exists\ a_{ji}>0,j\in\mathcal C \right\}.$$
Then, we can define $\mathcal{K}(\mathcal L):=\{i\in \mathcal{V}\ |\ \text{eq. }\eqref{eq:cont_s_i} \text{ holds}\}$, as the set of asymptotically contained nodes. For a given cardinality, say $d$, of the set $\mathcal L$, the partial containment control problem consists in finding optimal selection $\mathcal{L}^*(d)$ that maximizes the number of contained nodes, that is,
\begin{equation}
\label{prob:targ_cont}
\begin{aligned}
&\mathcal L^*(d)= \arg \max_{\mathcal L} \left| \mathcal{K}(\mathcal L)\right|\\
& \qquad \text{s.t.} \quad \left| \mathcal L \right| = d.
\end{aligned}
\end{equation} 
%An open question in containment control over signed graphs is to understand how the structural properties of a signed graph affect the ability to asymptotically contain the network. Inspired by \cite{loga15}, we propose the following \emph{containability index} to measure the readiness of a network to be contained:
%\begin{equation}
%c(\mathcal G)=\frac{\sum_{d=1}^{d_{\max}} \left| \mathcal K (\mathcal L^*(d)) \right| }{d_{\max}\left| \mathcal V \right|},
%\end{equation}
%where $d_{\max}$ is the maximum possible number of links connecting the leaders to the followers. The computation of this index requires the solution of problem \eqref{prob:targ_cont} $d_{\max}$ times. 
We observe that the numerical solution of this problem for $d>1$, although conceptually simple, would require to test for a number of alternative selections of the pinned nodes that is in the order of $n!$. An extensive search of the optimal solution is therefore computationally prohibitive even for relatively small networks. In what follows, we propose a computationally efficient heuristic approach to find a suboptimal solution of problem \eqref{prob:targ_cont}. 
%The algorithm is tested on a set of \emph{ad hoc} networks, built so that the optimal solution can still be computed.
\section{Convergence analysis}
%Here, we provide an algorithmic procedure to compute the steady-state values of the network nodes, and then give conditions for the achievement of asymptotic signed containment. 
Before giving our main results, we give  some relevant notation. Specifically, for the $h$-th SCC of the $l$-th level, we introduce the stack vector $x_{lh}$ of the states $\lbrace x_i \rbrace_{i\in \mathcal{V}^c_{lh}} $, and the vector
\begin{equation}
\label{eq:y}
y_{lh}(k):=\left[x_{lh}(k)^T, -x_{lh}(k)^T\right]^T.
\end{equation}
If $\mathcal{G}^c_{lh}$ is of type 2, $y_{hl}$ can be viewed as the vector containing all the states of the nodes in $\widetilde{\mathcal{V}}^c_{l \tilde{\mf h}}\cup  \widetilde{\mathcal{V}}^c_{l \tilde{\mf h}^* }$. From Lemma \ref{lem:str_bal}, $\widetilde{\mathcal{G}}_{lh}$ is composed by two disconnected SCCs. 
Therefore, we can find a permutation matrix $T_{lh}$ such that, defining $z_{lh}(k)=T_{lh} y_{lh}(k)=[z_{l \tilde{\mf h}}(k)^T z_{l\tilde{\mf h}^*}(k)^T ]^T$, we can write 
\begin{equation}\label{eq_zeta}
z_{lh}(k+1)= \begin{bmatrix} Z_{l\tilde{\mf h}} & 0 \\ 0 & Z_{l\tilde{\mf h}^*} \end{bmatrix} z_{lh}(k),
\end{equation} 
where $Z_{l\tilde{\mf h}}$ and $Z_{l\tilde{\mf h}^*}$ are the submatrices extracted from $\widetilde{A}$ associated to the nodes in $\widetilde{\mathcal{V}}_{l\tilde{\mf h}}^c$ and in $\widetilde{\mathcal{V}}^c_{l \tilde{\mf h}^*}$.
In what follows, for any node $h$ of level $l$ in $\mathcal{G}^c$ corresponding to a type 1 SCC of $\mathcal{G}$,
we indicate with $\xi_{lh}$ the left eigenvector associated to the unique eigenvalue $\lambda=1$ of block ${A}_{lh}$ of matrix $A$ in eq. \eqref{eq:signed_uncontrolled}, while, given a type 2 SCC $\mathcal{G}^c_{lh}$, we denote 
$\wt{\xi}_{l\tilde{\mf h}}$ ($\wt{\xi}_{l \tilde{\mf h}^*}$) the left eigenvector associated to the unique eigenvalue $\lambda=1$ of $Z_{l\tilde{\mf h}}$ ($Z_{l \tilde{\mf h}^*}$). 
%Furthermore, for each $\widetilde{\mathcal{G}}_{lh}^c$, we denote by $\widetilde \xi_{lh}$ and $\widetilde \psi_{lh}$ the left and the right eigenvectors associated to the unique eigenvalue $\lambda=1$ of ${\widetilde A}_{lh}$, respectively.
%

%Finally, let us consider two SCCs, say $\mathcal{G}_{lh}^c$ and $\mathcal{G}_{kj}^c$, with $\mathcal{G}_{kj}$ being of type 2: the submatrix $A_{l h, kj}$ extracted from $A$ that describes the (unidirectional) connections between the two SCCs can be decomposed as
%\begin{equation}
%\label{eq:dec_offdiag}
%A_{l h, kj}=\left[A_{l h, kj}^+\ A_{l h, kj }^-\right],
%\end{equation}
%where $A_{l h, kj}^ +$ ($A_{l h, kj}^ -$)is the matrix containing the first $|\mathcal{V}^s_{k\mf j} |$ (the last $|\mathcal{V}^s_{k\mf j^*} |$) columns of $\widetilde A_{l h, kj}$.

By exploiting the condensations introduced in Section \ref{subsec:cond}, here we explore the network level by level, to finally provide an algorithm that computes the steady-state configuration of any SCC in the graph.
Let us start by characterizing the asymptotic behaviors of the nodes in the RSCCs (i.e. in the level 1 of $\mathcal{G}^c$). 
\begin{theorem}\label{thm:1}
For all $h=1,\ldots,n_1^c$,
\begin{itemize}
\item if $\mathcal G_{1h}^c$ is of \emph{type 1}, then
\begin{equation}
\lim_{k\rightarrow+\infty}x_i(k)= \xi_{1 h}^T x_{1h}(0),\qquad  \ \forall i\in \mathcal{V}^c_{1h}\label{eq:rscc_consensus}
\end{equation} 
\item if $\mathcal G_{1h}^c$ is of \emph{type 2}, the SCC polarizes and
\begin{equation}\label{eq:rscc_neg_sb}
\begin{aligned}
\lim_{k\rightarrow+\infty}x_i(k)&=\wt{\xi}_{1 \tilde{\mathfrak h}}^Tz_{1 \tilde{\mathfrak h}}(0) \qquad\quad\ \ \forall i\in \mathcal{V}^s_{1 \mathfrak h}\\
\lim_{k\rightarrow+\infty}x_i(k)&=-\wt{\xi}_{1 \tilde{\mathfrak h}^*}^Tz_{1 \tilde{\mathfrak h}^*}(0) \qquad \forall i\in \mathcal{V}^s_{1 \tilde{\mathfrak h}^*}
\end{aligned}
\end{equation}
\item if $\mathcal G_{1h}^c$ is of \emph{type 3}, then
\begin{equation}
\label{eq:rscc_neg_nsb}
\lim_{k\rightarrow +\infty}{x}_{1h}(k)=0.
\end{equation}
\end{itemize}
\end{theorem}
\vspace{0.2cm}

\qquad \textit{Sketch of the proof:}
the proof exploits the fact that the dynamics of the nodes in an SCC of the first level are independent of the dynamics of the nodes not belonging to that SCC. Accordingly, the steady-state values of the nodes in type 1 SCCs can be studied using classic results on consensus, while that of the nodes in type 2 or type 3 SCCs can obtained leveraging results on balanced and unbalanced signed graphs, respectively.\qquad \qquad
{\small\QEDA}
\vspace{0.1cm}

Next, we define the upstream and the downstream of a node of a DAG.

\vspace{0.1cm}
\begin{define}\label{def:upstream}
For each node $i$ of a directed acyclic graph, its upstream (downstream) is the set of nodes, including $i$ itself, from which $i$ is reachable (which $i$ can reach) through a directed path. Moreover, we denote with $\Upsilon_{lh}^{\mathcal{G}^d}$ the upstream of the node $lh$ of $\mathcal{G}^d$.
\end{define}
For any node $lh$ of $\mathcal{G}^c$, $\delta_i(lh)$ is the number of nodes of the $i$-th level of $\mathcal{G}^c$ that are in the upstream of $lh$, for $i=1,\ldots,l-1$. Furthermore, we define the set $\mathcal{J}_{i}(lh):= \{j_1, \ldots, j_{\delta_i} \}$ as the set of nodes of level $i$ that are in the upstream of node $l h$, $i=1,\ldots,l-1$. Set $\mathcal{J}_{i}(lh)$ can be partitioned as follows:
\begin{equation*}
\mathcal{J}_{i}(lh)=\left\{\mathcal{J}_{i1}(lh),\mathcal{J}_{i2}(lh),\mathcal{J}_{i3}(lh)\right\},
\end{equation*}
where $\mathcal{J}_{it}(lh)=\left\{\alpha\in\mathcal{J}_{it}(lh)\ |\ \mathcal{G}_{i\alpha}^c \text{ is type } t \right\}$, $t=1,2,3$.

%For any node $l\tilde{\mf h}$ of $\wt{\mathcal{G}}^c$, $\delta_i$ is the number of nodes of the $i$-th level of $\wt{\mathcal{G}}^c$ that are in the upstream of $l\tilde{\mf h}$, for $i=1,\ldots,l-1$. Furthermore, we define the set $\mathcal{J}_{i}(l\tilde{\mf h}):= \{j_1, \ldots, j_{\delta_i} \}$ as the set of nodes of level $i$ that are in the upstream of node $l \tilde{\mf h}$, $i=1,\ldots,l-1$.

We now give an algorithmic procedure to compute the steady-state values of the states of the nodes belonging to a generic SCC of $\mathcal{G}$.

\vspace{0.1cm}
\begin{theorem}\label{thm:algo}
For all $l=2,\ldots,\ell$, $h=1,\ldots,n_l^c$, the steady-state values $\bar{x}_{lh}$ of the nodes in $\mathcal{G}_{lh}^c$ can be computed through the following algorithm
\begin{equation}\label{eq:algo}
\begin{aligned}
&\bar{x}_{1p}=\left\{ \begin{array}{ll}
\left[\begin{array}{l}
+\wt{\xi}^T_{1\tilde{\mathfrak p}}z_{1\tilde{\mathfrak p}}(0)\mathbf{1}_{|\mathcal{V}_{1\tilde{\mathfrak p}}^s|}  \\
-\wt{\xi}^T_{1\tilde{\mathfrak p}}z_{1\tilde{\mathfrak p}^*}(0)\mathbf{1}_{|\mathcal{V}_{1\tilde{\mathfrak p}^*}^s|}
\end{array}\right] & \text{if } \mathcal{G}_{1p}^c \text{ is of type 1}\\
\xi_{1p}^Tx_{1p}(0)\mathbf{1}_{|\mathcal{V}_{1p}^c|} & \text{if } \mathcal{G}_{1p}^c \text{ is of type 2}\\
0 & \text{if } \mathcal{G}_{1p}^c \text{ is of type 3}
\end{array}
\right., \\
&\qquad \forall p \in\mathcal{J}_{1}(hl),\\
&\bar{x}_{sp}=\left( I-A_{sp}\right)^{-1}\sum_{\lambda=1}^{s-1}\sum_{i\in\mathcal{J}_\lambda (sp)}A_{sp,\lambda i}\bar{x}_{\lambda i}, \\
&\qquad\forall s=2,\ldots,l, p\in\mathcal{J}_{s}(hl).
\end{aligned}
\end{equation}
\end{theorem}
\vspace{0.2cm}

\qquad \textit{Sketch of the proof:}
the algorithm initialization is a direct application of Theorem \ref{thm:1}, while the expression of $\bar{x}_{sp}$ can be derived by induction. \quad {\small \QEDA}
\begin{coro}
If $\cup_{h=1}^{n_1^c}\mathcal{V}_{1h}=\mathcal C$, then network \eqref{eq:signed_uncontrolled} is signed contained.
\end{coro}
\begin{proof}
the thesis directly follows from Theorems \ref{thm:algo}.
\end{proof}
The above corollary means that the network is signed contained if the leaders set $\mathcal{C}$ is connected to each of the $f$ SCCs of the graph of the followers, that is, the subgraph $\mathcal{F}$ induced by the node set $\mathcal{V}-\mathcal{C}$. This implies that, to guarantee signed containment, the number of outgoing edges $d$ from the leaders has to be equal or higher than $f$. The following corollary gives sufficient conditions guaranteeing asymptotic containment of a given SCC of $\mathcal{G}$.
\begin{coro}\label{cor:scc}
For all $l\hspace{-1mm}=\hspace{-1mm}2,\ldots,\ell$, $h\hspace{-1mm}=\hspace{-1mm}1,\ldots,n_c^l$, the $h$-th SCC of the $l$-th level of $\mathcal G$ is signed contained if $\cup_{k\in\Upsilon_{lh}^{\mathcal{G}^c}}\mathcal{V}_{1k}\subseteq\mathcal C$.
\end{coro}
\begin{proof}
The dynamics of the nodes in any SCC of the network are decoupled by those of the nodes that are not in its upstream. Then, the thesis follows from Theorem \ref{thm:algo}.
\end{proof}
In other words, this means that if the RSCCs of the upstream of the considered SCC are (a subset of) the network leaders, then the SCC is contained.
\section{An algorithm for control design}
Given a network topology $\mathcal{G}$, the nodes that will be asymptotically signed contained may be more than those of the SCCs fulfilling the assumption of Corollary \ref{cor:scc}. However, this will depend on the initial conditions of the RSCCs of $\mathcal{G}$ that are not the network leaders. Therefore, if one aims at finding the optimal solution for problem \eqref{prob:targ_cont}, then the knowledge of the initial conditions of all the followers would be necessary for the leaders. In absence of this information, a suboptimal solution maximizing the number of nodes that are \emph{guaranteed} to be signed contained can be found. Specifically, rather then solving problem \eqref{prob:targ_cont}, we will focus on finding the optimal solution of the following problem:
\begin{equation}
\label{prob:part_cont}
\begin{aligned}
\mathcal{\hat L}(d)&= \arg \max_{\mathcal L} \left| \phi(\mathcal L)\right|\\
\text{s.t.}&\\
\left| \mathcal L \right|& = d,
\end{aligned}
\end{equation} 
where $\phi(\mathcal{L})$ is the subset of nodes of $\mathcal{G}$ belonging to SCCs fulfilling the assumptions of Corollary \ref{cor:scc}. 

Contrary to problem \eqref{prob:targ_cont}, the solution of this problem does not require a full exploration of the feasible solutions, and can be translated into an integer linear program (ILP), by adapting the procedure presented in \cite{dega18}. Indeed, by introducing the condensation $\mathcal{F}^c$ of the subgraph $\mathcal{F}$ of the followers, the algorithm solving problem \eqref{prob:part_cont} consists of the following steps:
\begin{enumerate}[(a)]
\item build a new graph $\overline{\mathcal{G}}=\{\overline{\mathcal V},\overline{\mathcal E}\}$ as follows:
\begin{itemize}%[(a)]
\item add to $\overline{\mathcal{V}}$ the set of roots $r_i$ of $\mathcal{F}^c$ and all the non-roots $\gamma_i$ of $\mathcal{F}^c$ that are in the downstream of no more than $d$ roots $r_i$;
\item for all pairs $\gamma_i,r_j\in\overline{\mathcal{V}}$, add an edge $(\gamma_i,r_j)$ to $\overline{\mathcal E}$, with associated binary variable $y_{ij}$, if in $\mathcal{F}^c$, $\gamma_i$ is in the downstream of $r_j$; 
\item add an additional node, $\pi$, representing the leader set $\mathcal{C}$, and connect it to all the $r_j$ in $\overline{\mathcal{V}}$ by adding a set of edges $(r_j,\pi)$ to $\overline{\mathcal{E}}$, with associated binary variable $y_{j\pi}$;
\end{itemize}
\item associate to all edges of the graph $\bar{\mathcal{G}}$ the following weights:
\begin{itemize}
\item $w_{ij} = |\gamma_i|$, $\forall i$, that is, all edges entering the $i$-th node $\gamma_i$ have a weight equal to the number of nodes in the SCC $\gamma_i$;
\item $w_{j\pi} = |r_j|$, $\forall j$, that is, all edges entering the $j$-th root $r_j$ have a weight equal to the number of nodes in the SCC $r_j$;
\end{itemize}
\item solve the following ILP:
\begin{align}\label{eq:cdm}
& \underset{y}{\max} \sum_i\sum_j w_{ij}y_{ij} + \sum_j w_{j\pi}y_{j\pi}\\
& \text{s.t.} \nonumber\\ \label{eq:num_pinned}
& \sum_j y_{j\pi} = d\\ \label{eq:SCC_count}
& \sum_i y_{ij} \leq k^{\mathrm{out}}_j y_{j\pi} \qquad\ \, \forall j\\ \label{eq:all_r_pinned}
k_i^{\mathrm{in}} & \sum_j y_{ij} \leq \sum_{j | \exists y_{ij} }y_{j\pi} \quad\ \ \ \forall i\\
& y_{ij}, y_{j\pi} \in \lbrace 0,1 \rbrace \qquad\quad\ \forall i,j 
\end{align}
\end{enumerate}
where $k_i^{\mathrm{in}}$ and $k_i^{\mathrm{out}}$ are the in- and out-degree of the $i-th$ node of graph $\overline{\mathcal{G}}$, respectively.

Let us briefly illustrate the procedure outlined above. We first create a new graph $\overline{\mathcal{G}}$, whose nodes are either RSCC of the subgraph of  the followers, or SCCs in the downstream of such RSCCs. Each node representing a RSCC is connected to the SCCs in its downstream. Notice that we do not include any node representing a SCC that has more than $d$ RSCCs in its upstream, and thus cannot be guaranteed to be contained according to Corollary \ref{cor:scc}. 
Then, we add an extra node $\pi$ to $\overline{\mathcal{G}}$ representing the set of leaders, and we connect it to all nodes $r_i$ representing the RSCCs. Finally, we associate to each edge in $\overline{\mathcal{G}}$ a weight equal to the number of nodes in the (R)SCC it points to. The solution of the ILP in \eqref{eq:cdm}-\eqref{eq:all_r_pinned} is then equivalent to determine the RSCCs that have to be directly controlled, together with the corresponding SCCs that are guaranteed to be contained. Namely, SCC $\gamma_i$ is contained for all possible initial conditions if there exists a $j$ such that $y_{ij} = 1$, and RSCC $r_j$ will be directly controlled if $y_{j\pi} = 1$. Accordingly, the objective function to be maximized in \eqref{eq:cdm} represents the total number of nodes that we can guarantee to contain according to Corollary \ref{cor:scc}. The constraint \eqref{eq:num_pinned} guarantees that the directly controlled nodes are $d$, while  \eqref{eq:SCC_count} that all the contained nodes are in the downstream of (some of) the leaders. Finally, eq. \eqref{eq:all_r_pinned} imposes that the nodes of an SCC are contained only if a node in each of the RSCCs in their upstream is directly controlled by one of the leaders.
\begin{rem}\label{rem:final}
Notice that our algorithm only determines which RSCCs of $\mathcal{F}$ have to be connected to the set of leaders. Indeed, the selection of the specific node of each RSCC, and the leader connected to it, is indifferent to the objective function of problem \eqref{prob:part_cont}. Therefore, this selection will be performed randomly in the numerical example that follows. Clearly, the selection may indeed impact on both the convergence rate and on the width of the convex hull in which the followers are asymptotically contained. However, the investigation of these aspects goes beyond the scope of the present work.
\end{rem}
\subsection*{Numerical example}
\begin{figure}
\centering
\includegraphics[width=\columnwidth]{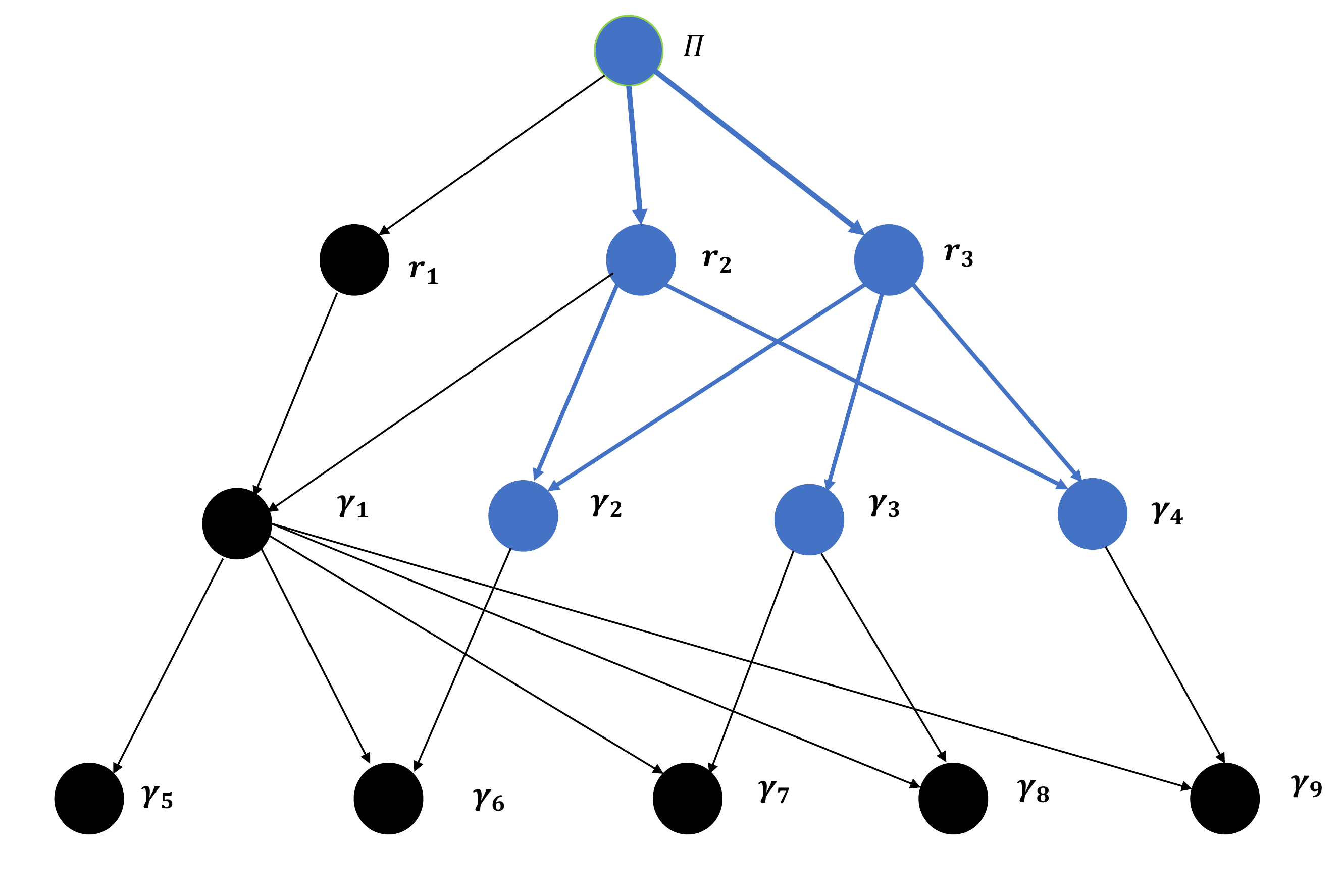}
\caption{Graph $\bar{\mathcal G}$ associated to the signed graph $\mathcal{G}$ in the numerical example. The (R)SCCs of $\mathcal{G}$ that are guaranteed to be asymptotically contained are depicted in blue, while the remaining (R)SCCs are in black.}\label{fig:opt_sel}
\end{figure}
We consider a signed graph $\mathcal{G}$ of $N=1500$ nodes distributed over $4$ levels and $15$ SCCs, whose dynamics follow equation \eqref{eq:signed_uncontrolled}. We assume that only $d=2$ control inputs can be deployed by the 3 leaders of the network. 
Following the steps of the algorithm, we first build the graph $\bar{\mathcal{G}}$, which is depicted in Figure \ref{fig:opt_sel}. Then, we solve the ILP  \eqref{eq:cdm}-\eqref{eq:all_r_pinned} and find that the leaders should directly control the RSCCs denoted by $r_2$ and $r_3$ in Figure \ref{fig:opt_sel} to maximize $|\phi|$, that is, the number of followers that are contained regardless of the initial conditions of the network. 
The optimum value of the objective function of problem \eqref{prob:part_cont}
is $\vert \phi(\hat{ \mathcal{L}}(2))\vert=508$. To validate our results, we simulated the system with the same leaders' states ($[-1, 0.5, 1]^T$) and two different sets of initial conditions, randomly selected from a uniform distribution in $\left[-20;20\right]$. In both cases, the nodes in $\phi(\hat{ \mathcal{L}}(2))$ (depicted in blue in Figure \ref{fig:state}) are asymptotically contained. Then, depending on the specific selection of the initial conditions, further nodes of the network might be asymptotically contained, as in the two simulations $\vert \mathcal{K}({ \mathcal{\hat L}}(2))\vert$ is equal to 516 and 843, respectively, see Figure \ref{fig:state}.

%the network is partially signed contained accordingly with Def. \ref{def:partial} but with a different $\vert \mathcal{Q} \vert$. To quantify this difference, let us define the increment
%that is the difference between the nodes effectively signed contained and the nodes guaranteed to be contained by the ILP \eqref{eq:cdm}-\eqref{eq:all_r_pinned}. The increment passes from $\delta_a=54\%$ to $\delta_b=14\%$ when the control goal of the leader is that to contain a network whose initial conditions are farther from the convex hull spanned by their states. The usefulness of our algorithm is even more evident from  the decreasing of $\delta$. When a leader set aim to control an unknown set of follower not only has a good estimation for $\vert \mathcal{K}({ \mathcal{L}}(d))\vert$ but also knows that as far the initial conditions of them are as $\vert \phi({\hat{ \mathcal{L}}}(d))\vert$ is close to $\vert \mathcal{K}({ \mathcal{L}}(d))\vert$.        

\begin{figure}
\centering
\subfigure{}{\includegraphics[width=0.8\columnwidth]{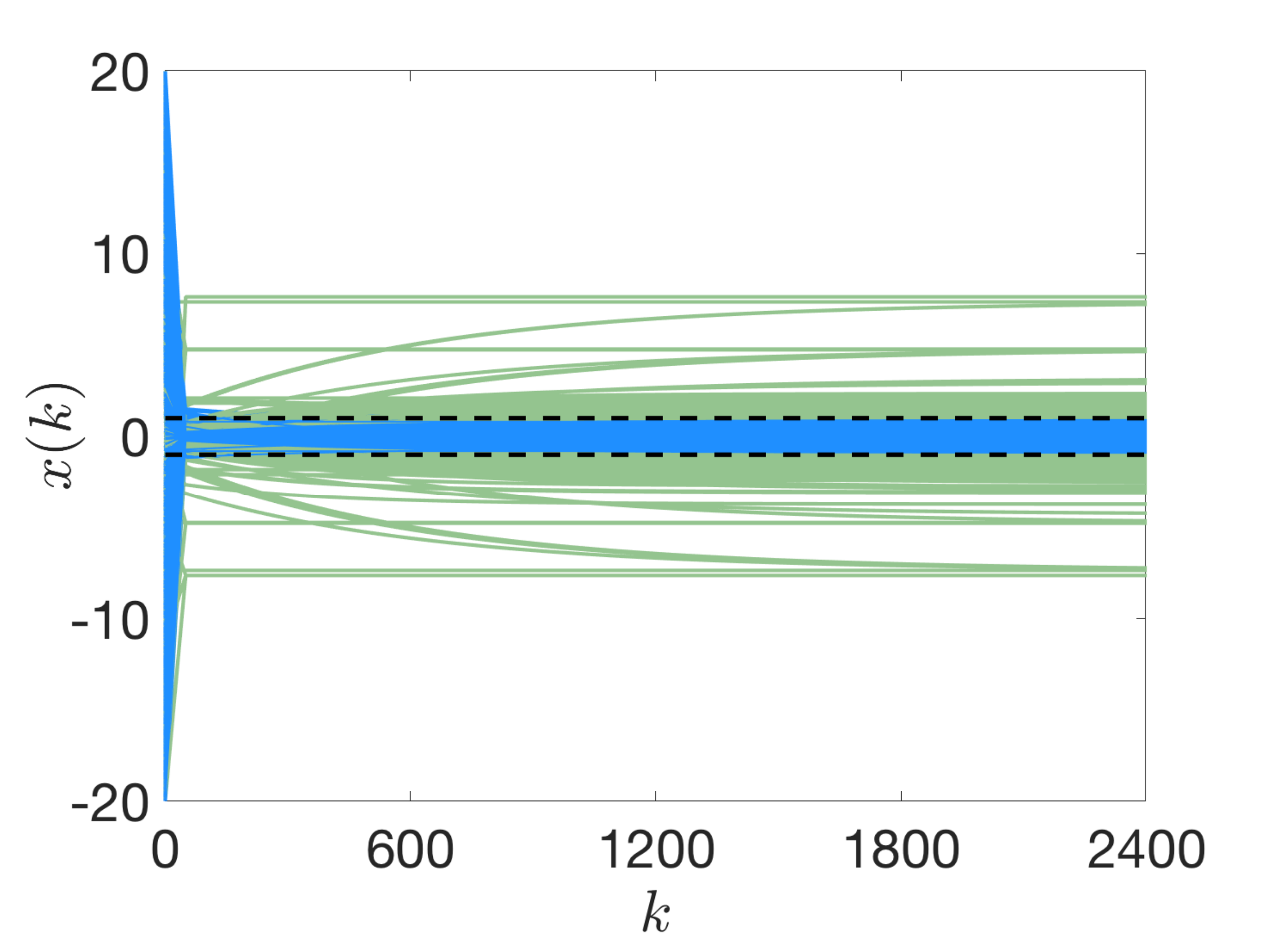}}\label{fig:syst1}
\subfigure{}{\includegraphics[width=0.8\columnwidth]{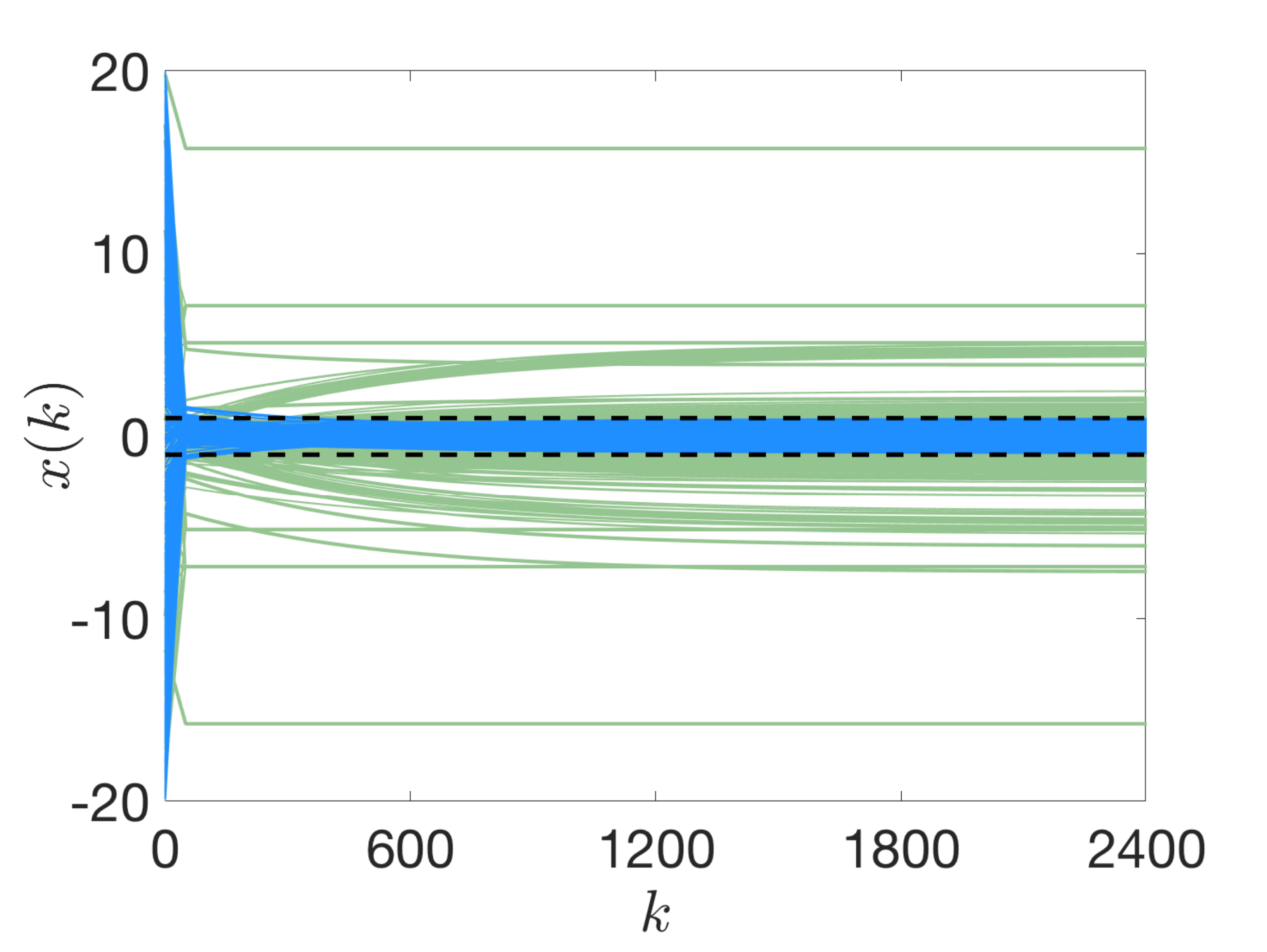}}\label{fig:syst1}
\caption{Two simulations of the network dynamics with leader states ${x}_{\mathcal{C}}\hspace{-1.5mm}=\hspace{-1.5mm}\left[ -1,0.5, 1 \right]^T$. The two simulations differ for the initial conditions of the followers, which are randomly selected from a uniform distribution in $\left[-10;10\right]$. The dotted black lines delimit the region where the leader aim at containing the followers, whose trajectories are in blue if they belong to SCCs fulfilling the assumptions of Corollary \ref{cor:scc}, while they are in green otherwise. In the top panel, the total number of asymptotically contained nodes is 516, while they are 843 in the bottom panel.}
\label{fig:state}
\end{figure}

\section{Conclusion}
In this paper, we tackled the containment control problem in a multi-agent discrete-time system where the interactions can be both cooperative and antagonistic. In particular, we focused on the case in which the containment of the entire network is prohibited by constraints on the number of control inputs the leaders can exert on the follower. The partial containment control problem was then defined as searching for the optimal deployment of the available control inputs so as to maximize the number of contained nodes. A preliminary graphical study, based on two alternative condensations of the original graph, allowed the derivation of the conditions guaranteeing the containment of the atomic element of a directed network, that is, a strongly connected component. Leveraging the convergence analyses, an algorithm for maximizing the number of followers we can guarantee to contain was built. Our solution strategy was translated into an integer linear program, and its effectiveness was demonstrated on a testbed examples. Future work will extend this analysis to alternative scenarios in which, for instance, the leaders may not cooperate and have contrasting goals.

%----------------------------------------------------------------------------------------
%	BIBLIOGRAPHY
%----------------------------------------------------------------------------------------

\bibliographystyle{IEEEtran}

%----------------------------------------------------------------------------------------

\end{document}